\documentclass[12pt]{amsart}

\usepackage{graphicx,color,verbatim}
\usepackage{geometry}
\usepackage{color, graphicx}
\usepackage[all]{xy}

\usepackage{times}
\usepackage[scaled=0.92]{helvet}

\swapnumbers
\theoremstyle{plain}
\newtheorem*{thmno}{Theorem}
\newtheorem{thm}{Theorem}[section]
\newtheorem{prop}[thm]{Proposition}
\theoremstyle{definition}
\newtheorem{df}[thm]{Definition}
\newtheorem{remark}[thm]{Remark}
\newtheorem{notation}[thm]{Notation}

\usepackage{amsfonts}
\usepackage{amssymb}

\def \smallmat #1 #2 #3 #4 {{\scriptstyle \begin{pmatrix} {#1} & {#2} \\ {#3} & {#4} \end{pmatrix}}}
\def \tinymat #1 #2 #3 #4 {\bigl( \begin{smallmatrix} {#1} & {#2} \\ {#3} & {#4} \end{smallmatrix} \bigr)}
\newcommand{\norm}[1]{\left| #1 \right|}
\newcommand{\lquot}[2]{#1 \backslash #2 }

\def\cA{\mathcal{A}}
\def\cE{\mathcal{E}}
\def\cR{\mathcal{R}}
\def\Q{\mathbf{Q}}
\def\CC{\mathbf{C}}
\def\AA{\mathbf{A}}
\DeclareMathOperator{\GL}{GL}
\def\tor{{\rm tor}}
\renewcommand{\hat}{\widehat}
\DeclareMathOperator{\Z}{\mathbf{Z}}
\DeclareMathOperator{\C}{\mathbf{C}}
\renewcommand{\Re}{\mathsf{Re}}
\renewcommand{\geq}{\geqslant}
\renewcommand{\leq}{\leqslant}
\newcommand{\ord}{\mathrm{ord}}
\DeclareMathOperator{\Res}{\mathrm{Res}}

\begin{document}

\title[Toroidal automorphic forms for number fields]{Toroidal automorphic forms, Waldspurger periods\\ and double Dirichlet series}
\author[G.~Cornelissen]{Gunther Cornelissen}
\address{Mathematisch Instituut, Universiteit Utrecht, Postbus 80.010, 3508 TA Utrecht, Nederland}
\email{g.cornelissen@uu.nl}
\author[O.~Lorscheid]{Oliver Lorscheid}
\address{Max-Planck-Institut f\"ur Mathematik, Postfach 7280, 53072 Bonn, Deutschland}
\email{oliver@mpim-bonn.mpg.de}
\subjclass[2000]{11F12, 11M06, 11M41, 11R42}

\begin{abstract} \noindent The space of toroidal automorphic forms was introduced by Zagier in the 1970s: a $\GL_2$-automorphic form is toroidal if it has vanishing constant Fourier coefficients along all embedded non-split tori. The interest in this space stems (amongst others) from the fact that an Eisenstein series of weight $s$ is toroidal for a given torus precisely if $s$ is a non-trivial zero of the zeta function of the quadratic field corresponding to the torus. 

In this paper, we study the structure of the space of toroidal automorphic forms for an arbitrary number field $F$. We prove that this space admits a decomposition into a subspace of Eisenstein series (and derivatives) and a subspace of cusp forms. The subspace of Eisenstein series is generated by all derivatives up to order $n-1$ of an Eisenstein series of weight $s$ and class group character $\omega$ for certain $n,s,\omega$, namely, precisely when $s$ is a zero of order $n$ of the $L$-series $L_F(\omega,s)$. The  subspace of cusp forms consists of exactly those cusp forms $\pi$ whose central $L$-value is zero: $L(\pi,1/2)=0$.

The proofs are based on an identity of Hecke for toroidal integrals of Eisenstein series and a result of Waldspurger about toroidal integrals of cusp forms combined with non-vanishing results for twists of $L$-series proven by the method of double Dirichlet series. \end{abstract}

\thanks{We thank Gautam Chinta for help with multiple Dirichlet series and Wee Teck Gan for his remarks on the proof of Theorem \ref{WaldTor}.}

\maketitle

\section{Introduction}

A classical theorem of Hecke (cf.\ Hecke \cite{Hecke} Werke p.\ 201) shows that on the modular curve $X(1)$, the integral of an Eisenstein series along a closed geodesic of discriminant $d>0$ is essentially the zeta function of the number field $\Q(\sqrt{d})$. As was observed by Don Zagier in \cite{Zagier}, the formula fits into a more general framework, where integrals of automorphic forms for global fields over tori (of any discriminant) evaluate to $L$-series. The approach in \cite{Zagier} is to define a space of so-called \emph{toroidal} automorphic forms by the vanishing of these integrals for varying tori, and the author calls for an (independent) understanding of the space of toroidal automorphic forms, after which one may hopefully apply the gained knowledge in combination with the generalization of Hecke's formula to deduce something about zeta functions. For example, if the irreducible subrepresentations of the space of toroidal automorphic forms are tempered, the Riemann hypothesis follows. This seems for now an elusive programme, but see \cite{CL} for a case study for some function fields. 

In this paper, we study the space of toroidal automorphic forms in its own right. We use our increased knowledge (compared to the 1970s, when \cite{Zagier} was written) about the decomposition of the space of automorphic forms for a general number field (by Franke), toroidal integrals of cusp forms (by Waldspurger in his study of the Shimura correspondence), and non-vanishing of quadratic twists (by the method of multiple Dirichlet series, essentially in the works of Friedberg, Hoffstein and Lieman) to prove the following theorem, which summarizes our main results. Note that in light of the previous paragraph, we avoid using any unproven hypothesis about the zeros of $L$-series. 

\begin{thmno}
The space of toroidal automorphic forms for an arbitrary number field decomposes into an Eisenstein part and a cuspidal part. More precisely, \begin{enumerate} 
\item \textup{(cf.\ Theorem \ref{Eis})} The Eisenstein part of the space of toroidal automorphic forms is spanned by all derivatives of Eisenstein series of weight $s_0 \in \C$ and class group character $\omega$, precisely up to the order of vanishing of the $L$-series corresponding to $\omega$ at $s_0$.
\item \textup{(cf.\ Theorem \ref{nores})} No nontrivial residues of Eisenstein series are toroidal. 
\item \textup{(cf.\ Theorem \ref{WaldTor})} The cuspidal part of the space of toroidal automorphic forms is spanned by those cusp forms $\pi$ for which the central value of its $L$-series vanish: $L(\pi,1/2)=0$. 
\end{enumerate}
\end{thmno}

We will use the next section to set up notation and give precise definitions of the spaces involved. 

\begin{remark}
How many derivatives of Eisenstein series will be toroidal? It is reasonable to expect that the only multiplicities in the zeros of $L$-functions of a number field $F$  arise from multiplicities in the decomposition of the regular representation of the Galois group of the normal closure of $F/\Q$ (following the Artin formalism of factorisation of $L$-series). The Rudnick-Sarnak theory of statistical distribution of zeros of principal primitive $L$-series (cf.\ \cite{RS}, \cite{RubS}, \S 5) indicates that the zeros of different such principal primitive $L$-series should be uncorrelated. For example, for $F=\Q$, all zeros of the Riemann zeta function are expected to be simple. 

But of course, as soon as $F/\Q$ is non-abelian, there will be such multiplicities arising from irreducible representations of the Galois group of higher dimension. Cf.\ also the possibility that a Galois extension $N/\Q$ contains two distinct subfields that are arithmetically equivalent (corresponding to two subgroups of the Galois group $G$ from which the trivial representation induces the same representation of $G$,  cf.\ \cite{Klingen}), hence have the same zeta function, whose zeros will then occur with multiplicity in the zeta function of $N$.
\end{remark}

\begin{remark} \label{direct} Our way of averaging toroidal integrals is a two-step method: by \emph{first} relating them to twists of $L$-series, and \emph{then} using standard techniques to average those. Is there a direct way to average toroidal integrals, e.g.\ by a Rankin-Selberg unfolding of an Eisenstein series twisted with toroidal integrals? We did not work this out. But for example, the toroidal integral of a holomorphic weight two cusp form $f$ over a torus corresponding to $\Q(\sqrt{d})$ for some $d>0$ is the integral of the differential form $\omega$ corresponding to $f$ over the closed geodesic $\Gamma_d$ in the modular curve corresponding to $d$. In this specific case, one wants to have an asymptotic result for a sum of `modular symbols' $\sum\limits_{d <X}\, \int\limits_{\Gamma_d} \omega$.  Compare with \cite{ChintaGoldfeld}, where Eisenstein series twisted by modular symbols are studied.  
\end{remark}

\begin{remark} There is an obvious generalization of $T$-toroidal automorphic forms on $\GL(n)$ for a maximal indecomposable torus $T\subset\GL(n)$. One may wonder whether the methods presented in this paper can be generalized to the $\GL(n)$-case. 

Part of the ingredients of our proof are already in the literature: on the one hand, Wielonsky \cite{Wielonsky} generalized Hecke's formula (Theorem \ref{thm_hecke} (\ref{hecke1})): toroidal integrals of Eisenstein series on $\GL(n)$ that are induced by a parabolic subgroup of type $(n-1,1)$ equal certain $L$-series; see Lachaud \cite{Lachaud} for a recent treatment. On the other hand, Friedberg, Hoffstein and Lieman \cite{FHL} introduced double Dirichlet series w.r.t.\ $n$-th order twists of Hecke $L$-series and showed that they admit a meromorphic continuation and satisfy a functional equation. 

However, we are not aware of a  generalization of Hecke's theorem \ref{thm_hecke} to \emph{all} Eisenstein series on $\GL(n)$,  or of Walds\-purger's work on toroidal integrals of cusp forms (\ref{WaldTor}) to $\GL(n)$. Maybe there is a way to circumvent the relation to $L$-series, compare the previous remark and the method in \cite{CO}. 
\end{remark}

\begin{remark}
In \cite{Lachaud1} (extended version in \cite{Lachaud}) and \cite{Lachaud2}, Lachaud ties up the theory of toroidal automorphic forms with Connes' trace form programme \cite{Connes} in the study of zeros of zeta functions.  
\end{remark}

\begin{remark} For global function fields, methods more akin to the geometric Langlands programme allow one to prove that the space of toroidal automorphic forms is finite dimensional, and one can control the linear relations between Eisenstein series in a very precise way, leading to an actual dimension formula for the Eisenstein part of the space of toroidal automorphic forms, cf.\ \cite{Lorscheid}. 
\end{remark}

\section{Definition of toroidal automorphic forms} 
\begin{notation} 
Let $F$ be a number field, $F_v$ be the completion at $v$, $\AA=\AA_F$ the adeles of $F$. Set $G=\GL(2)$ and let $Z$ be its center. Let $\cA$ be the space of automorphic forms for $G$ over $F$ with trivial central character. 
\end{notation}

\begin{notation} \label{not:tori}Let $T\subset G$ be a maximal non-split torus defined over $F$, $\chi_T$ the corresponding character on the idele class group
and $E=F_T$ the corresponding quadratic field extension of $F$. 
This means that there is a non square $d\in F$ such that $E/F$ is generated by a square-root of $d$ 
and that $T(F)$ is conjugated in $G(F)$ to the standard torus
$$ T_d(F) \ = \ \left\{ \smallmat a b {bd} a \in G(F)\right\} \;. $$ 
If $T=T_d$, then write $\chi_d$ for $\chi_T$.
\end{notation}

\begin{df}
The \emph{zeroth Fourier coefficient w.r.t.\ $T$} (or \emph{$T$-toroidal integral}) of an automorphic form $f \in \cA$ is defined to be the function 
$$ f_T(g):= \int\limits_{\lquot{T_F Z(\AA)}{T(\AA)}} f(tg) \, dt $$
for $g \in G(\AA)$. 
\end{df}

\begin{df} 
\mbox{ }
\begin{enumerate}
\item The \emph{space of $T$- or $E$-toroidal automorphic forms for $F$} is 
$$ \cA_\tor(T):=\cA_\tor(E):= \{ f \in \cA \, : \, f_T(g)=0, \, \forall g \in G(\AA) \}.$$
\item The \emph{space of toroidal automorphic forms for $F$} is 
$$ \cA_\tor:= \bigcap_{E/F} \cA_\tor(E), $$
where the intersection is taken over all quadratic field extensions $E/F$. 
\end{enumerate}
\end{df}
\begin{remark}
These definitions are independent of the choice of torus corresponding to $E/F$ since they are conjugacy invariant. 
\end{remark}
\begin{remark} There is also a definition of $T$-toroidal automorphic form for a split torus $T$, but one has to be careful, since the toroidal integral $f_T$ as defined above over a split torus $T$ can diverge. This can be taken care of by subtracting suitable parabolic Fourier coefficients before integrating (cf.\ \cite[\S 1.5]{Lorscheid} for the definition in the function field case). One could thus consider the space of automorphic forms that are $T$-toroidal for all maximal tori $T$, split or not. Due to the results of the present text and \cite[\S 6.2]{Lorscheid} (which transfer to the number field case), this space coincides with $\cA_\tor$, and we forgo describing the more involved theory for split tori.
\end{remark}

The space $\cA$ is an automorphic representation of $G(\AA)$ for right translation by $G(\AA)$. By \cite{Franke} applied to $\GL(2)$, the space $\cA$ decomposes into a direct sum of automorphic representations as follows: $$\cA=\cA_0\oplus\cE\oplus\cR,$$ 
where $\cA_0$ is the space of cusp forms, $\cE$ is the space generated by the derivatives of Eisenstein series and $\cR$ is generated by the residues of these Eisenstein series and their ``derivatives''. We will give the precise definitions of $\cE$ and $\cR$ in sections \ref{section:toroidal_Eisenstein} and \ref{section:torres}, respectively.

Multiplicity one holds for $GL(2)$, hence if $\pi$ is any subrepresentation of $\cA$, it inherits this decomposition, since it is determined by its isomorphism type. In order to investigate the space of toroidal automorphic forms $\cA_\tor$, which is an automorphic representation by its very definition, it thus suffices to investigate $$\cA_{0,\tor}:= \cA_\tor \cap \cA_0, \hspace*{1cm} \cE_\tor:=\cE \cap \cA_\tor \hspace*{2mm} \mbox{and} \hspace*{2mm} \cR_\tor:=\cR \cap \cA_\tor$$ separately, since 
$ \cA_\tor=\cA_{0,\tor}\oplus \cE_\tor \oplus \cR_\tor.$

\section{Toroidal integrals of Eisenstein series: a formula of Hecke}
\label{section:toroidal_Eisenstein}

\begin{notation} Let $\chi$ denote a character of the idele class group $I=F^\times\backslash\AA^\times$. We can write $\chi=\omega \cdot \norm\ ^{s_0-\frac12}$ for some finite order character $\omega$, and we will sometimes regard a function of $\chi$ as a function of $s_0$ (assuming $\omega$ to be fixed). We set $\Re(\chi)=\Re(s_0)-\frac12$. We also remark that the shift in $s_0$ by $-\frac12$ is in accordance with the usual convention in the adelic theory of Eisenstein series, putting the center of symmetry at $0$.  

We define the \emph{principal series}
$$ \mathcal P(\chi) \ = \ \bigl\{\, f: G(\AA)\stackrel{\text{smooth}}\longrightarrow\C\; :\; \forall \tinymat a b {} d ,\; g\in G(\AA), f(\tinymat a b {} d g ) = \chi(a/d) \norm{a/d}^{1/2} f(g)\, \bigr\}\,, $$ 
where a function $f:G(\AA)\to\C$ is smooth if it is smooth in the usual sense at archimedean places and locally constant at finite places.
Let $f\in\mathcal P(\chi)$ be embedded in a \emph{flat section} $f_\chi$ of the principal series, i.e.\ there exists a function $f_\chi(s)$ of $s\in \C$ such that $f_\chi(s)\in\mathcal P(\chi\norm\ ^s)$ with $f=f_\chi(0)$ and $f_\chi(s)(e)=f(e)$ for all $s\in\C$, where $e=\tinymat 1 0 0 1 $. Note that every $f\in\mathcal P(\chi)$ is embedded into a unique flat section. In the following, we will write $f=f_\chi(0)\in\mathcal P(\chi)$ to refer to this situation. We define the \emph{(completed) Eisenstein series} as
$$ E(g,f) \ = \ L(\chi^2,\tfrac12)\ \cdot \sum_{\gamma \in \lquot{B(F)}{G(F)}} f(\gamma g) $$
in terms of meromorphic continuation. Here $L(\chi^2,1/2)$ denotes the \emph{completed $L$-series}, i.e.\ including the factors at infinity. Note that $E(g,f)$ is defined for all $f=f_\chi(0)$ unless $\chi^2=\norm\ ^{\pm1}$, when the Eisenstein series has a simple pole. At these values of $\chi$ the residues of the Eisenstein series define automorphic forms, which will be investigated in section \ref{section:torres}---these values of $\chi$ characterize the cases where the principal series $\mathcal P(\chi)$ is \emph{not} irreducible as an automorphic representation of $G(\AA)$. 

Also note that the symbol ``$E$'' is now in use for both a field $E/F$ and a function $E(g,f)$, but this should cause no confusion. 
\end{notation}

We now compute the toroidal integrals of Eisenstein series. Statement (\ref{hecke1}) in the theorem below is the adelic formulation of a theorem of Hecke (\cite{Hecke}). In this formulation, it was first stated by Zagier \cite{Zagier}. 

\begin{thm}
 \label{thm_hecke}
 Let $T$ be a maximal torus in $G$ corresponding to the quadratic field extension $E/F$. Let $\chi_E:I\to\C^\times$ be the quadratic character whose kernel equals the norms of $\AA_E$. For every $f=f_\chi(0)\in\mathcal P(\chi)$, every $g\in G_\AA$ and every character $\chi:I\to\C^\times$, there exists a holomorphic function $e_T(g,f_\chi(s))$ of $s\in\CC$ with the following properties.
 \begin{enumerate}
  \item\label{hecke1} For all $\chi$ such that $\chi^2\neq\norm \ ^{\pm1}$,
       \begin{equation*} \label{split_toroidal_integral_of_Eisenstein_series}
         E_T(g,f) \ = \ e_T(g,f)\ L(\chi,\tfrac12)\ L(\chi\chi_E,\tfrac12) \;.
       \end{equation*} 
  \item\label{hecke2} For every $g\in G_\AA$ and $\chi:I\to\C^\times$, there is a $f\in\mathcal P(\chi)$ such that $ e_T(g,f) \ \neq \ 0 $.
   \end{enumerate}
\end{thm}

\begin{proof} The strategy of the proof is as follows: we first prove (\ref{hecke1}) for $\Re(\chi)$ sufficiently large, so we are in the region of absolute convergence. We rewrite the Eisenstein series conveniently as a certain adelic integral. Then a change of variables identifies its toroidal integral with a Tate integral for an $L$-series of $E$.

\medskip 

\emph{Rewriting the Eisenstein series.}\ First, we explain how to represent $E(g,f)$ by a certain adelic integral, and then the statement in (\ref{hecke1}) is a simple application of a change of variables. Let $\varphi$ be a Schwartz-Bruhat function on $\AA^2$. Then
\begin{equation}\label{eq:tate_integral}
  F(g,\varphi,\chi) = \int\limits_{Z(\AA)} \varphi((0,1)zg) \chi(\det zg)\norm{\det zg}^{1/2}\,dz 
\end{equation}
is a Tate integral for $L(\chi^2,1)$, which converges if the real part of $\chi$ is larger than $1/2$ (the square of the character $\chi$ occurs because in the above integrand $\chi(\det(z))=\chi(z)^2$). One verifies easily that $F(\cdot,\varphi,\chi)$ is an element of $\mathcal P(\chi)$. 

In \cite[Ch.\ VII, \S 7]{Weil}, Weil defines a particular test function $\varphi_0$ (the ``standard function'') with the property that for $e=\tinymat 1 0 0 1 $,we have $F(e,\varphi_0,\chi)=c_F^{-1} L(\chi^2,1)$ for a nonzero constant $c_F$ that only depends on the field $F$. Thus
$$ c_F\cdot F(g,\varphi_0,\chi) \ = \ L(\chi^2,1) \cdot f(g) $$
for an $f\in\mathcal P(\chi)$ with $f(e)=1$; in particular, $F(\cdot,\varphi_0,\chi)$ is a non-trivial element of $\mathcal P(\chi)$. 

We note that for every $g\in G(\AA)$, the function $\varphi_0(\,\cdot\,g)$ is a Schwartz-Bruhat function, too. Since $\mathcal P(\chi)$ is irreducible, the integrals $F(g,\varphi,\chi)$ for varying $\varphi$ exhaust all products of the form $L(\chi^2,1)f(g)$, where $f\in\mathcal P(\chi)$. Thus there is a Schwartz-Bruhat function $\varphi=\varphi(f)$ for every $f\in\mathcal P(\chi)$ such that
\begin{equation}\label{eq:f_vs_phi}
      E(g,f) \ = \ \sum_{\gamma \in \lquot{B(F)}{G(F)}} F(\gamma g,\varphi,\chi), 
\end{equation}
where the equality has to be interpreted in terms of meromorphic continuation. Since for all $\varphi$, the Tate integral \eqref{eq:tate_integral} is a multiple of $L(\chi^2,1)$, there exists $f\in\mathcal P(\chi)$ for every Schwartz-Bruhat function $\varphi$ such that equation \eqref{eq:f_vs_phi} holds true.

\medskip 

\emph{Computing the toroidal integral.}\ With this reformulation at hand, we can prove the theorem precisely as it has been done by Zagier in \cite{Zagier}: by identifying $T(F)$ with $E^\times$ and $T(\AA_F)$ with $\AA_E^\times$ and using Fubini's theorem, the toroidal integral of $E(g,f)$ is changed into a Tate integral for an $L$-series of $E$. Explicitly, 
\begin{equation*} E_T(g,f)=\int\limits_{\lquot{T(F)Z(\AA)}{T(\AA)}} \sum_{\gamma \in \lquot{B(F)}{G(F)}}\; \int\limits_{Z(\AA)} \varphi((0,1)z \gamma tg) \chi(\det z\gamma tg)\norm{\det z\gamma tg}^{1/2}\,dz  \, dt \end{equation*}
Via $T(F)=E^\times$ and $$\lquot{B(F)}{G(F)} = \lquot{F^\times}{F^2-\{(0,0)\}} = \lquot{F^\times}{E^\times},$$ we identify the ``domain of integration'' with $$\lquot{T(F)Z(\AA)}{T(\AA)} \times \lquot{B(F)}{G(F)} \times Z(\AA) \, \cong \AA_E^\times,$$
and note that this identification is compatible with measures. Hence
$$  E_T(g,f)= \chi(\det g) \norm{\det(g)}^{1/2} \cdot \int\limits_{\AA^\times_E} \Phi(t)\,  (\chi \circ N_{E/F})(t)\norm{t}_E^{1/2}\, dt $$
for a certain Schwartz-Bruhat function $\Phi$ on $\AA^\times_E$. This is a Tate integral for $L_E(\chi \circ N_{E/F},1/2)$
and factors in a product of two $L$-series  as in (\ref{split_toroidal_integral_of_Eisenstein_series}), thus giving (\ref{hecke1}) for $\Re(\chi)$ sufficiently large. 
Statement (\ref{hecke1}) now follows for all characters by meromorphic continuation in $s$ since both sides are meromorphic functions of $s\in\C$. 

\medskip 

\emph{Non-vanishing of the 'constant'.} \ For (\ref{hecke2}), note that the non-vanishing of $e_T(g,f)$ follows from the fact that we can choose $\varphi$ such that the test function $\Phi$ is again the standard function as described by \cite[Ch.\ VII, \S 7]{Weil}, and then \begin{equation} \label{e} e_T(g,f_\chi(s))= c_E^{-1} \chi(\det(g)) \norm{\det(g)}^{s+1/2}\end{equation} is a nonvanishing holomorphic function  of $s$. 
\end{proof}

In order to accord for possible multiple zeros of $L$-series, we need to also take into account higher derivatives of Eisenstein series. 
\begin{notation}
\mbox{ }
\begin{enumerate}
\item We denote by $E^{(n)}(g,f)$ the $n$-th derivative of $E(g,f_\chi(s))$ w.r.t.\ $s$ at $s=0$. 

\item We denote by $\cE$ the space of automorphic forms that is generated by all derivatives $E^{(n)}(\cdot,f)$ of Eisenstein series, where $n\geq0$ and $f\in\mathcal P(\chi)$ with $\chi$ varying through all idele class group characters whose square is not trivial. 

\item Similarly we denote by $L^{(n)}(\chi,1/2)$ the $n$-th derivative of $L(\chi,1/2+s)$ at $s=0$ and by $e_T^{(n)}(g,f)$ the $n$-th derivative of the function $e_T(g,f_\chi(s))$ at $s=0$. 
\end{enumerate}
\end{notation}

\begin{df}We say $\chi$ is a zero of $L(\cdot,1/2)$ of order $n$ if $L^{(i)}(\chi,1/2)=0$ for all $i<n$ but $\neq 0$ for $i=n$, and then we write $\ord_\chi L(\cdot,1/2)=n$.  
\end{df}

\begin{prop}
 \label{lemma_zagier_for_derivatives}
 Let $T$ be a maximal non-split torus in $G$ and $n$ a non-negative integer.
\begin{enumerate}
\item \label{wiet1}  For all $g\in G_\AA$ and $\chi$ such that $\chi^2\neq\norm \ ^{\pm1}$, we have
 $$ E_T^{(n)}(g,f) \ = \ 
    \sum_{\substack{i+j+k=n\\i,j,k\geq0}} \frac{n!}{i!\,j!\,k!}\ e_T^{(i)}(g,f)\ L^{(j)}(\chi,\tfrac12)\ L^{(k)}(\chi\chi_T,\tfrac12) \;. $$
    \item \label{wiet2} The $n$-th derivative ${E}^{(n)}(g,f)$ of an Eisenstein series is $E$-toroidal if and only if $\chi$ is a zero of $L(\cdot,1/2) \cdot L(\cdot \chi_E,1/2)$ of order at least $n$.
    \end{enumerate}
\end{prop}

\begin{proof}
The first part follows from the Leibniz rule. For (\ref{wiet2}), use (\ref{wiet1}) and observe the following:
\begin{itemize} 
\item The derivatives $e_T^{(i)}(g,f)$ are non-zero as function of $f$, as is easily seen from (\ref{e}).
\item If $E^{(m)}(g,f)$ is $E$-toroidal, then so is $E^{(n)}(g,f)$  for all $n<m$ since $E^{(m)}$ generates an automorphic subrepresentation of the space of toroidal automorphic forms that contains all derivatives of lower order. 
\end{itemize}
This finishes the proof. 
\end{proof}

\section{An application of double Dirichlet series: toroidal Eisenstein series}

In this section we will prove the following:

\begin{thm} \label{Eis} 
Let $ f \in \mathcal{P}(\chi)$. The $n$-th derivative ${E}^{(n)}(g,f)$ of an Eisenstein series is toroidal if and only if $\chi$ is a zero of $L(\cdot, \frac12)$ of order at least $n$.
Hence
$$ \mathcal{E}_{\mathrm{tor}}  = \langle E^{(n)}(\cdot,f)  \, : \, \exists \chi, f \in \mathcal{P}(\chi) \mbox{ and }  n \leq \ord_\chi L(\cdot,\tfrac12) \rangle. $$
\end{thm}

\begin{proof} Recall that we know from the computation of toroidal integrals in Proposition \ref{lemma_zagier_for_derivatives} when the $n$-th derivative ${E}^{(n)}(\cdot,f)$ of an Eisenstein series is $E$-toroidal.  It suffices to prove that for all $\chi$ there exists a quadratic $E/F$ such that $L(\chi \chi_E,1/2) \neq 0$. This follows from Theorem \ref{DD} below. \end{proof}

As before, we now write $\chi=\omega \cdot \norm\ ^{s_0-\frac12}$ for a finite character $\omega$, where we consider $s=s_0\in\C$ as varying parameter. We use the notation $L(\omega,s)$ for $L(\chi,1/2)$.

\begin{thm} \label{DD} Let $F$ denote a number field, let $\omega$ denote a class group character on $F$. Then there exists a quadratic field extension $E/F$ such that 
$L(\omega \chi_E,s) \neq 0. $
\end{thm}

\begin{remark} Before we start with the proof, we make some incomplete historical remarks. Non-vanishing of quadratic twists can be proven by sieve methods, but only for a restricted set of number fields. A method that works more uniformly is that of multiple Dirichlet series (so Fourier coefficients of metaplectic Eisenstein series); unfortunately, the result we need is not literally in the existing literature on multiple Dirichlet series, but rather arises from a combination of existing methods. We observe that for $\Re(s) \neq \frac12$, the  result can be proven using ``unweighted'' double Dirichlet series (see Chinta, Friedberg and Hoffstein \cite{CFHasymp} Thm.\ 1.1). To extend to the (for us interesting) range $\Re(s)=\frac12$, one needs to analytically continue the double Dirichlet series with weights; for higher order characters and a general number field, this is done by Friedberg, Hoffstein and Lieman in \cite{FHL}; and for quadratic twists of function fields by Fisher and Friedberg \cite{FF}. It is the methods of these latter two sources that we will combine. Since proofs of many facts are literally the same, we will not repeat them here, but we will set up all required notations.   One can go a (small) step further and combine the method with Tauberian theorems to establish lower bounds on the number of non-vanishing twists of bounded conductor, but we will not need those. Also, we refrain from discussing averaging of toroidal integrals \emph{directly}, without first relating them to $L$-series, cf.\ also Remark \ref{direct}. 
\end{remark}

\begin{proof}[Proof of Theorem \ref{DD}]
First note that the $L$-series we consider are `completed' by the correct archimedean factors, so they do not have trivial zeros outside the critical strip. 

Because of the functional equation, we can asume that $\frac12 \leq \Re(s) \leq 1$. Since $L(\omega \chi_E,s)$ does not vanish on $\Re(s)=1$ (\cite{MM} Ch.\ 1, \S 4), we can even assume $\frac12 \leq \Re(s) < 1$.

The strategy of the proof is now the following: we assume by contradiction that all non-trivial twists $L(\omega \chi_E, s)$ vanish at $s=s_0$. Then the (analytically continued) double Dirichlet series $Z^0_\omega(s,w)$ (to be defined below), with the trivial twist extracted, vanishes identically in $w$ for $s=s_0$. But it has residue at $w=1$ a non-zero constant times $L(\omega, 2s_0)$. Hence we find $L(\omega, 2s_0)=0$, and this is impossible if  $\frac12 \leq \Re(s_0) < 1$. 

To define the double Dirichlet series in a rigorous way, we follow \cite{FHL}, \S 1. Since the class number of $F$ is not necessarily one, the most natural double Dirichlet series (and the one that has a natural analytic continuation and set of functional equations) doesn't only sum over quadratic twists $\chi_E$, but rather over more general characters on a ray class group. We now introduce this series first. 

Let $S=S_f\cup S_\infty$ denote a finite set of places of $F$ that contains all infinite places $S_\infty$ of $F$ and a set $S_f$ of finite places such that the ring of $S_f$-integers has class number one. For $v$ a finite place corresponding to an ideal ${\mathfrak p}_v$, let $q_v=|\mathcal{O}/\mathfrak{p}_v|$. Set $C=\prod_{v\in S_f} \mathfrak{p}_v^{n_v}$, where $n_v=1$ for $v$ not above $2$, and for $v$ above $2$, $n_v$ is so large that any $a \in F_v$ with $\ord_v(a-1) \geq n_v$ is a square in $F_v$. Let $H_C$ denote the ray class group of modulus $C$, let $h_C:=|H_C|$ and set $$R_C = H_C \otimes \Z/2 = \Z/a_1 \times \dots \times \Z/a_r.$$ Choose generators $b_i$ for $\Z/a_i$, and choose a set $\mathcal{E}_0$ of ideals prime to $S$ that represent the $b_i$. For any $E_0 \in \mathcal{E}_0$, let $m_{E_0}$ denote an element of $F^*$ that generates the (principal) ideal $E_0 \mathcal{O}_S$. Let $\mathcal{E}$ denote a set of representatives of $R_C$ that are of the form $E=\prod E_0^{n_{E_0}}$,
where the $E_0$ are elements of $\mathcal{E}_0$ and the $n_{E_0}$ are natural numbers, and set $m_E=\prod m_{E_0}^{n_{E_0}}$ with the convention that $\mathcal{O} \in \mathcal{E}$ with $m_{\mathcal{O}}=1$. Then $E \mathcal{O}_S = (m_E)$. Let $I(S)$ denote the set of fractional ideals coprime to $S_f$. For $d,e \in I(S)$ coprime, write $d=(a)EG^2$ with $E \in \mathcal{E}, a \in F^*, a\equiv 1 \mbox{ mod } C, G \in I(S)$; and define $$\chi_d(e) = \left( \frac{d}{e} \right) := \left( \frac{am_E}{e} \right).$$ This is well-defined (cf.\ \cite{FHL}, Prop.\ 1.1), it does not depend on the decomposition of $d$ (but it does depend on the choice of $m_E$).
For $d$ principal, $\chi_d$ is the usual quadratic character for $F(\sqrt{d})/F$ (cf.\ Notation \ref{not:tori}). 

Let $L_S(\omega,s)$ denote the $L$-series of $F$ for the class group character $\omega$, but with the Euler factors corresponding to the places in $S$ removed. For $d \in I(S)$, let $S_d$ denote the set of primes above $d$. Let $J(S)$ denote the set of integral ideals in $I(S)$. For $d \in J(S)$, let $|d|$ denote its norm. Write $d=d_0 d_1^2$ with $d_0$ squarefree. 

We define the weight factor to be
$$ a(\omega,s,d):=\sum_{{e_i \in J(S)}\atop{e_1e_2 | d_1}} \frac{\mu(e_1) \chi_d(e_1) \omega(e_1 e_2^2)}{|e_1|^s |e_2|^{2s-1}}, $$
where $\mu$ is the M\"obius function. Let $\rho$ denote a character on the idele class group unramified outside $S$; this will be used later on to filter out principal ideals, which are the ones we are interested in. 
We define the double Dirichlet series as 
$$Z_{\omega,\rho}(s,w):=\sum_{d \in J(S)} \frac{L_{S \cup S_d}(\omega \chi_d,s) \rho(d)}{|d|^w}\cdot a(\omega,s,d). $$
This is convergent for $\Re(s)$ and $\Re(w)$ sufficiently large (say, $\geq 1$). 
The following properties are proven in exactly the same way as in \cite{FHL} (The only  difference to \cite{FHL}, which treats the case of twists by characters of higher order, is the set of functional equations, which here is of order 12 instead of 32, cf.\ \cite{FHL}, Remark 2.6.):
\begin{enumerate}
\item The function $Z_{\omega,\rho}(s,w)$ admits a meromorphic continuation to $\C^2$;
\item The poles of $Z_{\omega,\rho}(s,w)$ are located on the union of the lines $$w=0, w=1, s=0, s=1, w+s=\frac12 \mbox{ and } w+s=\frac32;$$ 
\item If $\rho \neq 1$, then $Z_{\omega,\rho}(s,w)$ is holomorphic at $w=1$; if $\rho=1$ and $s \neq 1/2$, $Z_{\omega,\rho}(s,w)$ has a simple pole at $w=1$. If $\rho=1$ and $s=\frac12$, $Z_{\omega,\rho}(s,w)$ has a double pole at $w=1$, and $L_S(\omega^2,2s)$ has a simple pole. We have $$ \lim\limits_{w \rightarrow 1} (w-1)Z_{\omega,\rho}(s,w) = \left\{ \begin{array}{ll} 0 & \mbox{ if } \rho \neq 1; \\ (\Res_{w=1} \zeta_F(w)) \cdot \left( \prod_{v \in S_f} \zeta_{F,v}(1)^{-1}\right) \cdot L_S(\omega^2,2 s) & \mbox{ if } \rho = 1, \end{array} \right.$$ 
\end{enumerate}
(cf.\ \cite{CFHsurvey}, Section 5; see also Section 5.3 in \emph{loc.\ cit.} for a computation of  the principal part of $Z_{\omega,\rho}(s,w)$ around $w=1$, which we will not need here.)

Let \begin{equation} \label{defZ0} Z^0_{\omega}(s,w):=\sum_{{d \in J(S)}\atop{[d]=0}} \frac{L_{S \cup S_d}(\omega \chi_d,s)}{|d|^w} \cdot a(\omega,s,d), \end{equation} denote the modified double Dirichlet series, where we only sum over principal ideals $d$ (indicated by the fact that their class $[d]$ is trivial in $R_C$).  Note that by plugging in the decomposition of the characteristic function of the class $[0]$ as $h_C^{-1} \sum\limits_{\rho \in \hat{R_C}} \rho$, we find
$$ Z^0_\omega(s,w) = \frac{1}{h_C} \sum_{\rho \in \hat{R_C}} Z_{\omega, \rho}(s,w). $$
Hence $Z^0_\omega(s,w)$ inherits an analytic continuation from $Z_{\omega,\rho}(s,w)$. 
Using the above computation of residues, we find that \begin{equation}\label{resD} \lim\limits_{w \rightarrow 1} (w-1)Z^0_\omega(s_0,w) = c \cdot L(\omega^2,2s_0), \end{equation}
where $c$ is some non-zero constant. 

We can now finish the proof of the theorem. We are assuming that all `principal' twists $L(\omega \chi_d, s_0)$ ($[d]=0$, $\chi_d$ non-trivial) vanish at some $s_0$ with $\frac12 \leq \Re(s_0) < 1$.  Note first that this obviously implies the vanishing of all twists for the modified $L$-series with the $S \cup S_d$-Euler factors removed. 

A slight complication arises since the $s_0$ we consider are outside of the region of absolute convergence of the series $Z^0_\omega(s,w)$, but we can use a convexity estimate to get that for such $s_0$ and $\Re(w)$ \emph{large enough}, the double Dirichlet series $Z^0_\omega(s_0,w)$ as defined in (\ref{defZ0}) will also converge. This is because of Phragm\'en-Lindel\"of estimates in the $d$-aspect of the form $$|L_{S \cup S_d}(\omega \chi_d,s_0)a(\omega,s_0,d)| \ll |d|$$ (cf.\ \cite{CFHsurvey}, 3.3), so $\Re(w)>2$ will do. 

Now recall our hypothesis that $L_{S \cup S_d}(\omega \chi_d, s_0)=0$ for all principal $d$ with $\chi_d \neq 1$, i.e., $d$ not a square. Hence if we substract the terms for which $d=e^2$ is a square from (\ref{defZ0}), we find the identically zero function \begin{equation} \label{fam} Z^0_{\omega}(s_0,w)- \sum_{{e \in J(S)}\atop{[e^2]=0}} \frac{L_{S\cup S_{e}}(\omega,s_0)}{|e|^{2w}}\cdot a(\omega,s_0,e^2) = 0 \end{equation}  (first for $\Re(w)$ sufficiently large, hence after analytic continuation, for all $w$).

We now prove that the term we have subtracted off doesn't have a pole at $w=1$; actually, it converges absolutely at $w=1$.  Write the term as \begin{equation} \label{brr}  L_S(\omega,s_0) \cdot \sum_{{e \in J(S)}\atop{[e^2]=0}} \frac{B_{e}}{|e|^{2w}}\cdot a(\omega,s_0,e^2), \end{equation} where $B_e$ consists of the reciprocals of the (finitely many) $(S_{e}\setminus S)$-Euler factors of $L(\omega,s_0)$. Essentially, the absolute convergence is due to the fact that the exponent of $|e|$ is $\approx 2$ if $w \approx 1$ and the other factors are small in $|e|$. We present some details. First of all, note that $L_S(\omega,s_0)$ doesn't have a pole for $\frac12 \leq \Re(s_0) < 1$. 
We estimate for $\Re(s_0) \geq 1/2$ $$\norm{B_e} \ = \ \norm{\prod_{\mathfrak{p}\in S_e\setminus S} \left( 1-\frac{\omega(\mathfrak{p})}{|\mathfrak{p}|^{s_0}} \right)} \ \leq \ 2^{\varpi(|e|)[F:\Q]} \ \leq \ d(|e|)^{[F:\Q]} \ \ll \ |e|^{1/3}$$ (where $\varpi(n)$ is the number of positive prime divisors of an integer $n$, and $d(n)$ is the number of positive divisors of $n$, e.g.\ \cite{HW} Section 22.13). We estimate the other factor as $$ \norm{a(\omega,s_0,e^2)} \ = \ \norm{\sum_{{e_1,e_2\in J(S)}\atop{ e_1e_2|e}} \frac{\mu(e_1)\omega(e_1e_2^2)}{|e_1|^{s_0}|e_2|^{2s_0-1}}} \ \leq  \sum_{{e_1,e_2\in J(S)}\atop{ e_1e_2|e}} 1 \ \leq  \ d(|e|)^{2[F:\Q]} \ \ll  \ |e|^{1/3}, $$ since $\norm{\mu(e_1)\omega(e_1e_2^2)}\leq1$ and $ |e_1|^{s_0}|e_2|^{2s_0-1} \geq 1$ for $\Re(s_0) \geq \frac12$. We combine this into the estimate $$\sum_{{e \in J(S)}\atop{[e^2]=0}} \norm{\frac{B_{e}}{|e|^{2w}}\cdot a(\omega,s_0,e^2)}\quad \ll \quad k\sum_{{e \in J(S)}\atop{[e^2]=0}} \norm{\frac{1}{|e|^{2w-2/3}}}, $$ for some constant $k$, the latter sum being an absolutely convergent series for $w=1$.

We conclude from this and equation (\ref{fam}) that $Z^0_\omega(s_0,w)$ also doesn't have a pole at $w=1$, i.e.,  $ \lim\limits_{w \rightarrow 1} (w-1)Z^0_\omega(s_0,w)=0$. Then by (\ref{resD}), we find that $L(\omega^2,2s_0)=0$ with $1 \leq \Re(2s_0)<2$, which is impossible. This finishes the proof. 
\end{proof}

\section{Toroidal residues of Eisenstein series}
\label{section:torres}

Let $\chi$ be a character of the idele class group such that $\chi^2=\norm \ ^{\pm1}$ and $f=f_\chi(0)\in\mathcal P(\chi)$. Then the Eisenstein series $E(g,f_\chi(s))$ has a simple pole at $s=0$, but the residue $$R(g,f):=\Res_{s=0} E(g,f_\chi(s))$$ is an automorphic form. More generally, we consider the ``derivatives''  
$$ R^{(n)}(g,f) := \lim_{s\to0}\frac{d^n}{ds^n} \left( s\cdot E(g,f_\chi(s)) \right) $$
for $n \geq 0$.

\begin{df}
Define $\mathcal R$ as the space of automorphic forms that is generated by the functions $R^{(n)}(g,f)$, where $n\geq0$, $f$ ranges through $\mathcal P(\chi)$ and $\chi$ ranges through all characters on the idele group such that  $\chi^2=\norm \ ^{\pm1}$.
\end{df}

\begin{thm} \label{nores}
$\cR_\tor=\{0\}$. 
\end{thm}

\begin{proof}
We compute the toroidal integral of $R(g,f)$ for a torus $T$ corresponding to a quadratic field  $E$:
 \begin{align*}
  R_T(g,f) &= \lim_{s\to0}\ s\ E_T(g,f_\chi(s)) \\
    &= \lim_{s\to0}\ s\ e_T(g,f_\chi(s))\ L(\chi,s+\tfrac12)L(\chi \chi_E,s+\tfrac12) \\
    &= e_T(g,f)\ \Res_{s=0}\ L(\chi,s+\tfrac12)L(\chi \chi_E,s+\tfrac12). 
 \end{align*}
 
 Recall that $\chi^2=\norm \ ^{\pm 1}$, hence $\chi \norm \ ^ {\mp 1/2}$ is quadratic, so either trivial or by class field theory equal to $\chi_E$ for some quadratic field extension $E/F$. In the case that $E$ is non-trivial, we have that $$L(\chi \chi_E,s+\tfrac12)=\zeta_F(s+\tfrac12 \pm \tfrac12),$$ where $\zeta_F$ is as usual the completed zeta function of $F$ with poles at $0$ and $1$, so $$\zeta_F(s+1/2 \pm 1/2)$$ has a pole at $s=0$ (for both choices of sign). Hence,  by the above formula for $R_T(g,f)$, the toroidal integral of $R_T(g,f)$ for this $E$ cannot vanish, since the other factor $L(\chi,1/2)$ does not vanish. 

If $\chi \norm \ ^ {\mp 1/2}$ is trivial, then we choose an arbitrary non-split torus $T$. We find that $$L(\chi,s+1/2)=\zeta_F(s+1/2 \pm 1/2)$$ has a pole at $s=0$, but the other factor in the above computation of the toroidal integral, $$L(\chi \chi_E,s+1/2)=L(\chi_E,s+1/2\pm1/2)$$ doesn't vanish at $s=0$.

This also implies that $R^{(n)}(g,f)$ cannot be toroidal, since $R(g,f)$ is contained in the automorphic representation generated by this automorphic form. \end{proof}

\section{An application of Waldspurger periods: toroidal cusp forms}

In this section, we prove the following:

\begin{thm}\label{WaldTor}
 A cuspidal representation $\pi\subset\cA_0$ is toroidal if and only if $L(\pi,1/2)=0$:  
 $$ \cA_{0,\tor} = \langle \pi \in \cA_0 \ : \ L(\pi,\tfrac12)=0 \rangle. $$
\end{thm}

\begin{proof}
 A formula of Waldspurger (\cite[Prop.\ 7]{Waldspurger2}) shows that the $T$-period of an automorphic form $f$ in an irreducible cuspidal representation $\pi$ is a non-zero multiple of 
 \[
  L(\pi,\tfrac12)\cdot L(\pi\otimes\chi_T,\tfrac12).
 \]
 Thus $\pi$ is toroidal if $L(\pi,1/2)=0$.

 We are left to prove the reverse implication. Assume that $L(\pi,1/2)\neq 0$. Note that all automorphic representation of $\GL(n)$ with trivial central character are self-contragredient. Thus, the functional equation of the $L$-series of $\pi$ in $s=1/2$ is
 \[
  L(\pi,1/2) \quad = \quad \epsilon(\pi,1/2) \ \cdot \ L(\pi,1/2),
 \]
 and necessarily $\epsilon(\pi,1/2)=1$. This allows us to apply a theorem of Friedberg and Hoffstein: let $\pi_v$ be representations of $G(F_v)$ such that $\pi \simeq \otimes'\pi_v$, where $v$ ranges over all places; let $S$ be the (finite) set of places $v$ such that $\pi_v$ is square integrable; let $\xi$ be the trivial Hecke character. Then \cite[Thm.\ B (1)]{FH} states that there are infinitely many different nonconjugate tori $T$ such that $L(\pi\otimes\chi_T,1/2)\neq0$ and such that for all $v\in S$, the local character $\chi_v$ is trivial, i.e.\ $T_v$ is split. In particular, there is such a non-split torus $T$.

 We want to apply to apply \cite[Thm.\ 2, p.\ 221]{Waldspurger2} (in the ``situation globale'', where the quaternion algebra is chosen to split), which implies that $\pi$ is not $T$-toroidal. To do so, we have to verify that condition (i) in loc.\ cit.\ is satisfied. By \cite[Lemme 8 (iii)]{Waldspurger2}, condition (i) is satisfied for all local factors $\pi_v$ that are not square integrable. If $v\in S$, then $T_v$ is split and \cite[Lemme 8 (ii)]{Waldspurger2} implies that condition (i) holds for square integrable $\pi_v$.
 
 This shows that $\pi$ is not toroidal, which concludes the proof of the theorem.
\end{proof}

\begin{remark}
Theorem {\cite[Thm.\ 4, p.\ 288]{Waldspurger3}} of Waldspurger says
that there exists a $\chi_d$ for $d\in F^\times$ such that $L(\pi\otimes\chi_d,1/2)\neq0$, but 
without the claim that $\chi_d$ is nontrivial. However, if there exists one such $d$ (square or not), there exist infinitely many, as may be seen from Lemma 7.1 in \cite{bbcfh}, which shows that the Dirichlet series occuring as Mellin transform of the series $t(\sigma)$ on p.\ 289 of the proof in \cite{Waldspurger3} cannot have only finitely many non-zero coefficients.
\end{remark}

\begin{remark}
There do exist number fields $F$ for which there exist non-trivial toroidal cusp forms. We only need $\pi$ to be a cusp form with $L(\pi,1/2)=0$. This happens when the root number is $-1$, which is for example the case for a cuspidal lift of a classical holomorphic cusp form to an imaginary quadratic field  (cf.\ also Waldspurger \cite{Waldspurger3}, Remark on p.\ 282). The argument also shows that there are no toroidal cusp for for $F=\Q$. 
\end{remark}

\bibliographystyle{plain}

\end{document}